\documentclass[a4paper,12pt]{amsart}
\usepackage{amsmath,amssymb,amsfonts,amsthm}
\usepackage{latexsym,mathrsfs}
\usepackage{graphicx}
\usepackage[all]{xy}
\input xypic
\usepackage{color}
\usepackage[pagebackref]{hyperref}
\hypersetup{colorlinks=true,linkcolor=red,citecolor=blue}
\usepackage[OT2,T1]{fontenc}
\usepackage[utf8]{inputenc}
\usepackage{lmodern}
\usepackage{microtype}
\usepackage{enumerate}
\usepackage{tikz-cd}
\usepackage{bm}
\usepackage{ulem}
\usepackage{mathtools}
\usepackage{tensor}
\theoremstyle{plain}
\newtheorem{theorem}[subsection]{{\bf Theorem}}

\newtheorem{corollary}[subsection]{{\bf Corollary}}
\newtheorem{proposition}[subsection]{{\bf Proposition}}
\newtheorem{lemma}[subsection]{{\bf Lemma}}
\theoremstyle{definition}

\theoremstyle{remark}

\numberwithin{equation}{subsection}

\DeclareMathOperator{\Hom}{Hom}

\DeclareMathOperator{\id}{id}

\DeclareBoldMathCommand{\bbot}{\bot}

\DeclareSymbolFont{cyrletters}{OT2}{wncyr}{m}{n}
\DeclareMathSymbol{\Sha}{\mathalpha}{cyrletters}{"58}

\begin{document}
\title[Uniform exponent bounds]{Uniform exponent bounds for integral group homology}
\author{Primo\v z Moravec}
\address{{
Faculty of  Mathematics and Physics, University of Ljubljana,
and Institute of Mathematics, Physics and Mechanics,
Slovenia}}
\email{primoz.moravec@fmf.uni-lj.si}
\thanks{ORCID: \url{https://orcid.org/0000-0001-8594-0699}.}
\thanks{The author acknowledges the financial support from the Slovenian Research and Innovation Agency (ARIS), research core funding No. P1-0222, and project No. J1-50001.}
\date{\today}
\begin{abstract}
We prove that, in each fixed degree, the exponent of the integral
homology of a finite group is bounded solely in terms of the degree and
the exponent of the group.  The proof combines the solution of the
restricted Burnside problem with a representability property of the bar
construction and may be viewed as a torsion analogue of the method of
acyclic models.  We also use the Lyndon--Hochschild--Serre spectral
sequence to obtain explicit bounds for finite solvable and nilpotent
groups in terms of their derived length and nilpotency class.
\end{abstract}
\maketitle

\section{Introduction}
\label{s:intro}

\noindent
Let $G$ be a finite group and $H_n(G,\mathbb{Z})$ its $n$-th integral homology group. It is well known \cite[Chapter III, Corollary 10.2]{Bro82} that $|G|\cdot H_n(G,\mathbb{Z})=0$. It is proved in \cite{Moravec2007} using Zel'manov's solution of the restricted Burnside problem \cite{Zel91a,Zel91b} that given $e\ge 1$, there exists a constant $c=c(e)$ such that $c\cdot H_2(G,\mathbb{Z})=0$ for all (locally) finite groups $G$ of exponent $e$. Subsequently, considerable work has been devoted to obtaining good upper bounds for $\exp H_2(G,\mathbb{Z})$. 
A recent survey and further references are given in
 \cite{Thomas2021}.

The purpose of this paper is to generalize the above-mentioned result of \cite{Moravec2007} to higher homology groups:

\begin{theorem}
\label{thm:uniform-homology-exponent}
For every pair of positive integers $n$ and $e$ there exists an integer
$N_n(e)$, depending only on $n$ and $e$, such that
$$N_n(e)H_i(G,\mathbb Z)=0$$
for every finite group $G$ of exponent dividing $e$ and every
$1\leq i\leq n$.
In particular,
$\exp H_n(G,\mathbb Z)$
is bounded in terms of $n$ and $\exp(G)$ only.

One may take
$$ N_n(e)=\prod_{q=1}^n |R(q,e)|,$$
where $R(q,e)$ stands for the $q$-generator restricted Burnside group of exponent $e$.
\end{theorem}

The argument uses the solution of the restricted
Burnside problem together with a degreewise representability property
of the bar complex. The proof is reminiscent of the method of acyclic models \cite[Chapter VIII]{Mac63}, but the
models $R(q,e)$ are not acyclic.  Instead, at the $q$-th stage one uses
the fact that $H_q(R(q,e),\mathbb Z)$ has a controlled annihilator.
It may therefore be viewed as a torsion version of the acyclic-model
argument.

The resulting bound stated in Theorem \ref{thm:uniform-homology-exponent} is highly non-effective, since the orders
of the restricted Burnside groups are in general
enormous, see Vaughan-Lee and Zel'manov \cite{VZ99}. Better bounds are obtained when other invariants of the groups in question are allowed to be included in the estimates. We briefly present a method based on the Lyndon--Hochschild--Serre spectral sequence. We provide bounds for $\exp H_n(G,\mathbb{Z})$ in the cases when $G$ is solvable or nilpotent, in terms of the derived length or nilpotency class, respectively. 
The filtration argument used in the Lyndon--Hochschild--Serre bound follows an approach already taken by Hilton \cite{Hil01}, and is also
closely related to Browder's work \cite{Browder1983}.  
\section{Uniform bounds for higher homology}
\label{sec:higher-homology-exponent}

\noindent
In this section we show that, in each fixed degree, the exponent of the
integral homology of a finite group is bounded in terms of the exponent
of the group alone. 

Fix a positive integer $e$, and denote by $\mathcal C_e$ the category of
finite groups of exponent dividing $e$.  

For $q\geq 1$, let
$R(q,e)$
denote the restricted Burnside group of rank $q$ and exponent $e$, that is
$$  R(q,e)=F_q/\bigcap_N N,$$
where $F_q$ is a free group of rank $q$, and the intersection is taken over all normal subgroups
$N\trianglelefteq F_q$ such that $F_q/N$ is finite and has exponent
dividing $e$. 
The group $R(q,e)$ is finite, by the solution of the restricted Burnside
problem \cite{Zel91a,Zel91b}, and its order depends on $e$ and $q$ only. Every finite $q$-generator group of exponent dividing $e$ is an image of $R(q,e)$. 

Throughout the rest of this section, the exponent $e$ is fixed, and we use the shorthand notation $R_q=R(q,e)$.
Fix a generator set $x_1,\ldots,x_q$ of $R_q$, and denote $m_q=|R_q|$. 

We use the standard inhomogeneous bar complex; see \cite{Bro82}.
Given a set $X$, let $\mathbb{Z}[X]$ be the free abelian group on $X$. For a group $G$, denote by $C_\bullet(G)$ its inhomogeneous bar complex:
\begin{align*}
    C_0(G) &= \mathbb{Z},\\
    C_q(G) &= \mathbb{Z}[G^q],
\end{align*}
where $q>0$.
The basis element corresponding to the tuple $(g_1,\ldots,g_q)\in G^q$ in $C_q(G)$ is written in the bar notation as $[g_1\mid\cdots\mid g_q]$.  The differential $\partial_q:C_q(G)\to C_{q-1}(G)$ is given by
\begin{equation}
    \begin{split}
    \label{eq:diff}
  \partial_q[g_1\mid\cdots\mid g_q]
  ={}&
  [g_2\mid\cdots\mid g_q]
  \\
  &+
  \sum_{i=1}^{q-1}
  (-1)^i
  [g_1\mid\cdots\mid g_i g_{i+1}
  \mid\cdots\mid g_q]
  \\
  &+
  (-1)^q
  [g_1\mid\cdots\mid g_{q-1}].
\end{split}
\end{equation}
In particular, $\partial_1=0$.  The homology of this complex is
$H_\bullet(G,\mathbb Z)$. If $f:G\to H$ is a homomorphism, then the induced map $C_\bullet(G)\to C_\bullet(H)$ will be denoted by $f_{\#}$.

\begin{lemma}
\label{lem:burnside-represents-tuples}
For every $G\in\mathcal C_e$ and every $q\geq 1$, the map
\[
  \Hom(R_q,G)
  \longrightarrow G^q,
  \qquad
  \varphi\longmapsto
  \bigl(\varphi(x_1),\ldots,\varphi(x_q)\bigr),
\]
is a bijection.  Moreover, these bijections are natural in $G$.
Consequently,
\[
  C_q(-)
  \cong
  \mathbb Z\bigl[
    \Hom(R_q,-)
  \bigr]
\]
as functors $\mathcal C_e\to\mathbf{Ab}$.
\end{lemma}

\begin{proof}
A homomorphism $R_q\to G$ is determined by the images of the
generators $x_1,\ldots,x_q$, so the displayed map is
injective.
Conversely, a tuple $\mathbf{g}=(g_1,\ldots,g_q)\in G^q$ determines a unique homomorphism $\phi_{\mathbf{g}}:R_q\to G$ sending $x_i$ to $g_i$ for all $i$.
This proves the bijection. 

Let $f:G\to H$ be an arbitrary homomorphism in $\mathcal{C}_e$. The above bijections commute with the maps induced by $f$, that is, the diagram
$$
\begin{tikzcd}
    \Hom(R_q,G) \arrow[r, "\sim"] \arrow[d, "f_*"'] & G^q \arrow[d, "f^q"] \\
    \Hom(R_q,H) \arrow[r, "\sim"'] & H^q
\end{tikzcd}
$$
is commutative.
The bijections in question are therefore natural in $G$.
Applying the free abelian group functor to the natural bijection
$\Hom(R_q,G)\cong G^q$
gives
$$\mathbb Z\bigl[
    \Hom(R_q,G)
  \bigr]
  \cong
  \mathbb Z[G^q]
  =
  C_q(G).$$
The result is proved.
\end{proof}

Let $u_q=[x_1\mid\cdots\mid x_q]\in C_q(R_q)$. Let $G$ be a finite group in $\mathcal{C}_e$.
If $\mathbf g=(g_1,\ldots,g_q)\in G^q$,
then $(\phi_{\mathbf g})_{\#}(u_q)=[g_1\mid\cdots\mid g_q]$.
Given $c\in C_r(R_q)$, the rule 
  $[g_1\mid\cdots\mid g_q]
  \mapsto
  (\phi_{\mathbf g})_\#(c)$
extends linearly to a natural transformation
$C_q(-)\rightarrow C_r(-)$.

\begin{proof}[Proof of Theorem \ref{thm:uniform-homology-exponent}]
Recall that
  $m_q=|R_q|$
and define
$$
  N_0=1,
  \qquad
  N_q=m_qN_{q-1}.
$$
Thus,
$$
  N_q=\prod_{j=1}^q |R(j,e)|.
$$
We prove inductively that, for every $q\geq 1$, there exist natural
homomorphisms
$h_i^{(q)}:C_i(-)\rightarrow C_{i+1}(-)$,
$1\leq i\leq q$,
such that, with $h_0^{(q)}=0$,
\begin{equation}
    \label{eq:homot}
      \partial_{i+1}h_i^{(q)}+h_{i-1}^{(q)}\partial_i=
  N_q\id_{C_i}
\end{equation}
for every $1\leq i\leq q$.
In other words, multiplication by $N_q$ is naturally null-homotopic in positive degrees through degree $q$.

We first treat the base case $q=1$.  Here
$R_1=\langle x_1\rangle\cong C_e$ and $m_1=e$.  In the unnormalized
bar complex of $R_1$, put
$$
  c_1=\sum_{j=0}^{e-1}[x_1^j\mid x_1]\in C_2(R_1).
$$
Using
$\partial_2[a\mid b]=[b]-[ab]+[a]$,
we obtain
$$
  \partial_2c_1
  =\sum_{j=0}^{e-1}
    \bigl([x_1]-[x_1^{j+1}]+[x_1^j]\bigr)\\
  =e[x_1],$$
because $x_1^e=1$.  For $G\in\mathcal C_e$ and $g\in G$, let
$\phi_g:R_1\to G$ be the homomorphism satisfying
$\phi_g(x_1)=g$, and define
$$
  h_1^{(1)}(G)[g]=(\phi_g)_\#(c_1).
$$
This construction is natural in $G$, and
$
  \partial_2h_1^{(1)}(G)[g]
  =e[g]
  =N_1[g].
$
This proves \eqref{eq:homot} for $q=1$.

Suppose the assertion is proved through degree $q-1$. To simplify notation, write $h_i=h_i^{(q-1)}$, where $i<q$. Consider the natural endomorphism $\alpha_q:
  C_q(-)\rightarrow C_q(-)$
defined by
\begin{equation}
    \label{eq:alpha}
      \alpha_q= N_{q-1}\id_{C_q}-h_{q-1}\partial_q.
\end{equation}
We claim that $\alpha_q$ takes values in cycles. For $q=1$ this follows from $\partial_1=0$.  If $q\geq 2$, the induction
hypothesis in degree $q-1$ gives
\begin{equation}
      \partial_qh_{q-1}+h_{q-2}\partial_{q-1}=N_{q-1}\id_{C_{q-1}}.
      \label{eq:cycl}
\end{equation}
Composing \eqref{eq:cycl} on the right with $\partial_q$ and using
$\partial_{q-1}\partial_q=0$
gives
$$\partial_qh_{q-1}\partial_q=N_{q-1}\partial_q.$$
Hence,
$$\partial_q\alpha_q=N_{q-1}\partial_q-
  \partial_qh_{q-1}\partial_q
  =0,$$
which proves the assertion, that is, $\alpha_q(G)(C_q(G))\subset Z_q(G)$ for every group in $\mathcal{C}_e$.
In particular, $z_q=\alpha_q(R_q)(u_q)$
belongs to $Z_q(R_q)$. By \cite[Chapter III, Corollary 10.2]{Bro82} we have that $m_q[z_q]=0$ in $H_q(R_q,\mathbb{Z})$. Therefore, there exists $c_q\in C_{q+1}(R_q)$ such that
$$\partial_{q+1}c_q=m_qz_q=m_q\alpha_q(R_q)(u_q).$$
Define $t_q(G):C_q(G)\to C_{q+1}(G)$ via
$$t_q(G)[g_1\mid\cdots\mid g_q]=(\phi_{\mathbf{g}})_\#(c_q).$$
Because this construction comes from the fixed $c_q\in C_{q+1}(R_q)$, the maps $t_q(G)$ are natural in $G$. We have that
\begin{align*}
    \partial_{q+1}t_q(G)[g_1\mid\cdots\mid g_q] &=
    \partial_{q+1}(\phi_{\mathbf{g}})_\#(c_q)\\
    &= (\phi_{\mathbf{g}})_\#(\partial_{q+1}c_q)\\
    &= m_q\cdot (\phi_{\mathbf{g}})_\#(z_q)\\
    &= m_q\cdot (\phi_{\mathbf{g}})_\#(\alpha_q(R_q)(u_q)).
\end{align*}
As $\alpha_q$ is a natural transformation, we have that
$$(\phi_{\mathbf{g}})_\#(\alpha_q(R_q)(u_q))=\alpha_q(G)(\phi_{\mathbf{g}})_\#(u_q)=\alpha_q(G)[g_1\mid\cdots\mid g_q].$$
This, together with the calculation above, shows that
\begin{equation}
    \label{eq:claim2}
    \partial_{q+1}t_q(G)=m_q\alpha_q(G).
\end{equation}
Now we define
$$h_q^{(q)}=t_q$$
and
$$h_i^{(q)}=m_qh_i$$
for $1\le i<q$.
For $i<q$, the induction hypothesis gives
\begin{align*}
      \partial_{i+1}h_i^{(q)}+h_{i-1}^{(q)}\partial_i
  &=
  m_q
  \left(
    \partial_{i+1}h_i
    +
    h_{i-1}\partial_i
  \right)
  \\
  &=
  m_qN_{q-1}\id_{C_i}
  \\
  &=
  N_q\id_{C_i}.
\end{align*}
In degree $q$, the equation \eqref{eq:claim2} gives
\begin{align*}
  \partial_{q+1}h_q^{(q)}
  +
  h_{q-1}^{(q)}\partial_q
  &=
  \partial_{q+1}t_q
  +
  m_qh_{q-1}\partial_q
  \\
  &=
  m_q\alpha_q
  +
  m_qh_{q-1}\partial_q
  \\
  &=
  m_q
  \left(
    N_{q-1}\id_{C_q}
    -
    h_{q-1}\partial_q
    +
    h_{q-1}\partial_q
  \right)
  \\
  &=
  N_q\id_{C_q}.
\end{align*}
This completes the induction proof of \eqref{eq:homot}.

Now fix $n$ and let $1\leq i\leq n$.  If $z\in Z_i(G)$,
then $\partial_i z=0$, and \eqref{eq:homot}, with $q=n$, yields
$$
  N_n z
  =
  \partial_{i+1}h_i^{(n)}(z)
  +
  h_{i-1}^{(n)}(\partial_i z)
  =
  \partial_{i+1}h_i^{(n)}(z).$$
Thus, $N_nz$ is a boundary, and hence
$N_n[z]=0$ in $H_i(G,\mathbb Z)$.
Therefore,
$N_nH_i(G,\mathbb Z)=0$
for every $1\leq i\leq n$, as required.
\end{proof}

We note that
the choice
$m_q=|R(q,e)|$
in the above proof is merely convenient.  The only property required of
$m_q$ is $m_q H_q(R(q,e),\mathbb Z)=0$.
Consequently, one may instead take
  $m_q=
  \exp H_q(R(q,e),\mathbb Z)$,
which gives the formally sharper bound
$$
  \exp H_n(G,\mathbb Z)
  \mid
  \prod_{q=1}^n
  \exp H_q(R(q,e),\mathbb Z).
$$
The same argument applies, without change, to locally finite groups
of exponent dividing $e$.  


\section{Exponent bounds from the Lyndon--Hochschild--Serre spectral sequence}
\label{sec:LHS-exponent}

\noindent
We record a general consequence of the Lyndon--Hochschild--Serre
spectral sequence which gives useful exponent bounds for the homology
of group extensions.  Throughout this section all homology groups are
taken with integral coefficients.

Let
$$
  1\longrightarrow N\longrightarrow G\longrightarrow Q\longrightarrow 1
$$
be a group extension.  The homological Lyndon--Hochschild--Serre
spectral sequence is
  $$E^2_{p,q}=
  H_p\bigl(Q,H_q(N,\mathbb Z)\bigr)
  \Longrightarrow
  H_{p+q}(G,\mathbb Z),$$
see, for example, \cite[Chapter VII, Section 6]{Bro82}.
Here the action of $Q$ on $H_q(N,\mathbb Z)$ is the one induced by
conjugation in $G$.

The following standard filtration argument is closely related to
 \cite{Hil01}:

\begin{proposition}
\label{thm:LHS-exponent-bound}
Let
\[
  1\longrightarrow N\longrightarrow G\longrightarrow Q\longrightarrow 1
\]
be an extension of finite groups, and let $n\geq 1$.  Then
  $\exp H_n(G,\mathbb Z)$
  divides
  $$\exp H_n(Q,\mathbb Z)
  \prod_{j=1}^n
  \exp H_j(N,\mathbb Z).$$
\end{proposition}

\begin{proof}
Consider the Lyndon--Hochschild--Serre spectral sequence
$E^2_{p,q}$ as above. Denote
$a_j=\exp H_j(N,\mathbb Z)$
and
$b_n=\exp H_n(Q,\mathbb Z)$.
For $q>0$, the underlying abelian group
$H_q(N,\mathbb Z)$ is annihilated by $a_q$.  Hence,
$a_q E^2_{p,q}=0$
for
$p\geq0$ and $q\geq1$.
For $q=0$ we have
$H_0(N,\mathbb Z)\cong\mathbb Z$
with trivial $Q$-action, and therefore
$E^2_{p,0}=H_p(Q,\mathbb Z)$.
In particular,
$b_nE^2_{n,0}=0$.
Every term on every subsequent page is a subquotient of the
corresponding term on the preceding page.  Consequently, an
annihilator of $E^2_{p,q}$ also annihilates $E^\infty_{p,q}$.
It follows that
$a_qE^\infty_{n-q,q}=0$
for $1\leq q\leq n$,
and
  $b_nE^\infty_{n,0}=0$.

Convergence of the spectral sequence gives a finite filtration
$$
  0=F_{-1}
  \subseteq F_0
  \subseteq\cdots
  \subseteq F_n
  =
  H_n(G,\mathbb Z)
$$
whose associated graded groups satisfy
$$
  F_p/F_{p-1}
  \cong
  E^\infty_{p,n-p}
  \qquad
  (0\leq p\leq n).
$$
The quotient
$F_p/F_{p-1}$
is annihilated by $a_{n-p}$ whenever $0\leq p<n$, while 
the last quotient
$F_n/F_{n-1}$
is annihilated by $b_n$. We immediately deduce that
  $b_n a_1a_2\cdots a_n\,
  H_n(G,\mathbb Z)=0$.
\end{proof}

Proposition \ref{thm:LHS-exponent-bound} therefore becomes particularly
simple when the kernel is abelian.

\begin{corollary}
\label{cor:abelian-kernel-homology}
Let
\[
  1\longrightarrow A\longrightarrow G\longrightarrow Q\longrightarrow 1
\]
be an extension of finite groups with $A$ abelian.
Then, for every $n\geq1$,
$\exp H_n(G,\mathbb Z)$
divides
$(\exp A)^n\exp H_n(Q,\mathbb Z)$.
\end{corollary}

As a first structural application, one obtains an explicit bound for
metabelian groups.

\begin{corollary}
\label{cor:metabelian-homology}
Let $G$ be a finite metabelian group of exponent dividing $e$.  Then
for every $n\geq1$, the exponent of 
$H_n(G,\mathbb Z)$
divides $e^{n+1}$.
\end{corollary}

\begin{proof}
    Apply Corollary~\ref{cor:abelian-kernel-homology} to
$$
  1\longrightarrow G'\longrightarrow G
  \longrightarrow G/G'\longrightarrow1.
$$
Both $G'$ and $G/G'$ are abelian and have exponent dividing $e$, hence the result follows.
\end{proof}

The same argument can be iterated along the derived series.

\begin{corollary}
\label{cor:solvable-homology-exponent}
Let $G$ be a finite solvable group of exponent dividing $e$ and
derived length at most $d$.  Then, for every $n\geq1$,
$\exp H_n(G,\mathbb Z)$ divides
$e^{\,1+n(d-1)}$.
\end{corollary}

\begin{proof}
    We argue by induction on $d$.  The case $d=1$ is straightforward.  If $d>1$, then
$A=G^{(d-1)}$ is abelian and $G/A$ has derived length at most $d-1$.
Corollary~\ref{cor:abelian-kernel-homology} and the induction hypothesis
give that
  $\exp H_n(G,\mathbb Z)$
  divides
$e^n\exp H_n(G/A,\mathbb Z)$, which divides
  $e^n e^{1+n(d-2)}=e^{1+n(d-1)}$.
\end{proof}

An analogous argument gives a bound in terms of nilpotency class.

\begin{corollary}[See \cite{Hil01}, Lemma A.11]
\label{cor:nilpotent-homology-exponent}
Let $G$ be a finite nilpotent group of exponent dividing $e$ and
nilpotency class at most $c$.  Then, for every $n\geq1$,
$\exp H_n(G,\mathbb Z)$
divides
$e^{\,1+n(c-1)}$.
\end{corollary}

\begin{proof}
We argue by induction on $c$.  The case $c=1$ is the abelian case.
If $c>1$, then $A=\gamma_c(G)$ is abelian and $G/A$ has nilpotency
class at most $c-1$.  Applying
Corollary~\ref{cor:abelian-kernel-homology} and then the induction
hypothesis gives
$\exp H_n(G,\mathbb Z)
  \mid e^n e^{1+n(c-2)}
  =e^{1+n(c-1)}$.
\end{proof}


\end{document}